\documentclass[11pt]{amsart}
\usepackage{amsmath}
\usepackage{amsfonts}
\usepackage{amssymb}
\newtheorem{thm}{Theorem}[section]

\newtheorem{cor}[thm]{Corollary}

\newtheorem{remark}[thm]{Remark}
\newtheorem{lemma}[thm]{Lemma}
\newtheorem{prop}[thm]{Proposition}

\newtheorem{exam}[thm]{Example}
\newtheorem{defn}[thm]{Definition}

\newcommand{\bb}[1]{\mathbb{#1}}
\newcommand{\cl}[1]{\mathcal{#1}}

\usepackage[all]{xy}

\begin{document}

\title[The SOH Operator System]{The SOH Operator System}

\author[W.~H.~Ng and V.~I.~Paulsen]{Wai Hin~Ng and Vern I.~Paulsen}
\address{Department of Mathematics, University of Houston,
Houston, Texas 77204-3476, U.S.A.}
\email{rickula@math.uh.edu, vern@math.uh.edu}

\thanks{Research supported in part by NSF grant DMS-1101231}
\subjclass[2000]{Primary 46L15; Secondary 47L25}

\begin{abstract}
In this paper we examine a natural operator system structure on Pisier's self-dual operator space. We prove that this operator system is completely order isomorphic to its dual with the cb-condition number of this isomorphism as small as possible. We examine further properties of this operator system and use it to create a new tensor product on operator systems.
\end{abstract}

\maketitle


\section{Introduction} 

Pisier\cite{Pi1} proved that, for each dimension, there is a unique operator space with the property that it is completely isometrically isomorphic to its dual space. In this paper we study the analogous problem in the matrix ordered setting. Since the dual of a matrix ordered space is still a matrix ordered space, it is natural to ask if a matrix ordered space is completely order isomorphic to its dual. 

Unlike the operator space case,  there are many operator systems that are completely order isomorphic to their matrix-ordered dual.  Since the dual of an operator system also carries a matrix norm, it is natural to ask if an operator system is ever simultaneously completely order isomorphic and completely norm isomorphic to its dual. We will show that this is impossible.  In fact, we will prove that any complete order isomorphism between an operator system and its dual has a cb-condition number that is bounded below by 2.

We will see that for the many standard examples of finite dimensional operator systems that are completely order isomorphic to their duals, the cb-condition number  of this order isomorphism grows unbounded as the dimension tends to infinity.

We will then create a ``natural'' operator system from Pisier's $OH(n)$ spaces, that we denote by $SOH(n)$ and show that these operator systems have the property that there exists a map from the space to its dual that is a complete order isomorphism and has cb-condition number of exactly 2.

We then explore some further properties and applications of the operator systems $SOH(n).$ We prove that subsystems and quotients of $SOH(n)$ are completely order isomorphic to $SOH(m)$ for some $m \le n.$

Finally, we use ``approximate cp-factorization through SOH'' to create a new tensor product on operator systems and examine some of its properties.

\section{Operator System and Operator Space Duality}

We assume that the reader is familiar with the basic definitions and properties of operator spaces, operator systems, completely bounded and completely positive maps.  For more details the reader should see the books \cite{Pa2, Pi2}. We only review the basic definitions of duals of operator spaces and operator systems, since these are the objects that we wish to contrast.



If $V$ is an operator space, then the space of bounded linear functionals on $V$, denoted $V^d$, comes equipped with a natural {\bf dual matrix-norm}. Briefly, a matrix of linear functionals $F= (f_{i,j}) \in M_n(V^d)$ is identified with a linear map $F:V \to M_n$ and we set $\|(f_{i,j})\|_n = \|F\|_{cb}.$

Recall that given a $*$-vector space $V,$ the vector space $M_n(V)$ is also a  $*$-vector space with $*$-operation given by $(v_{i,j})^* = (v_{i,j}^*)^t$ where $^t$ denotes the transpose. By a {\bf matrix order} on $V$ we mean a family of cones of self-adjoint elements, $C_n \subseteq M_n(V)_h,$ that satisfy:
\begin{enumerate}
\item $C_n \cap (-C_n) = \{ 0\},$
\item $M_n(V)$ is the complex span of $C_n,$
\item if $A=(a_{i,j}) \in M_{n,m}$ is a matrix of scalars and $(v_{i,j}) \in C_n,$
then $A^*(v_{i,j}) A = ( \sum_{k,l} \overline{a_{i,k}} v_{k,l} a_{l,j}) \in C_m.$
\end{enumerate}
We call such a $*$-vector space a {\bf matrix-ordered space} and simplify notation, when possible, by setting $C_n = M_n(V)^+.$ Note that if $V_1 \subseteq V$ is a $*$-invariant vector subspace, then the cones $C_n \cap M_n(V_1)$ endow $V_1$ with a matrix-order that we call the {\bf subspace order,} or more simply, we refer to $V_1 \subseteq V$ as the {\bf matrix ordered subspace.} 

Given two matrix ordered spaces $V$ and $W$ we call a map $\phi:V \to W$ {\bf completely positive} provided that $\phi^{(n)}:M_n(V) \to M_n(W)$ is positive for all $n.$

Given a matrix-ordered space $V,$ we let $V^{\ddag}$ denote the vector space of all linear functionals on $V.$ Given a linear functional $f:V \to \bb C,$ if we let $f^*: V \to \bb C$ be the linear functional $f^*(v) = \overline{f(v^*)},$ then this makes $V^{\ddag}$ a $*$-vector space. We identify an $n \times n$ matrix of linear functionals $(f_{i,j})$ with the linear map, $F:V \to M_n$ defined by $F(v) = (f_{i,j}(v)),$ and set $M_n(V^{\ddag})^+$ equal to the cone of completely positive maps. Then this gives a sequence of cones on the dual that satisfy properties (1) and (3), but not generally (2).  
When $V$ is also a normed space, then we let $V^d$ denote the space of bounded linear functionals on $V,$ which is a subspace of $V^{\ddag}$ and is endowed with the subspace order.

However, when $V$ is an operator system, then $V^d$ endowed with this set of cones is a matrix-ordered space and we refer to this as the {\bf matrix-ordered dual of $V$.}

The easiest way to see that these cones span, is to use Wittstock's decomposition theorem \cite{Wi, Pa2} which says that the completely bounded maps on an operator system are the complex span of the completely positive maps. 

Since every operator system $V$ is also an operator space, its dual comes equipped with two structures, an operator space structure and a matrix-order structure. We wish to focus on the contrast between these two structures.

We begin with some examples.  We always identify the dual of $\bb C^n$ with $\bb C^n$ again via the map that sends the standard basis $\{ e_j \}$ to the dual basis $\{ \delta_j \}.$ 

\begin{exam} The identification of  $\ell^{\infty}_n$ with the continuous functions on an $n$ point space makes $\ell^{\infty}_n$ into an operator system with $\sum_j A_j \otimes e_j \in M_m(\ell^{\infty}_n)^+$ iff $A_j \in M_m^+$ for all $j.$ Moreover, a map $\Phi: \ell^{\infty}_n \to M_m$ with $\Phi(e_j) = A_j$ is completely positive iff $A_j \in M_m^+$ for all $j.$  From this is follows that the map $e_j \to \delta_j$ is a complete order isomorphism between $\ell^{\infty}_n$ and $(\ell^{\infty}_n)^{d}.$ Thus, as a matrix-ordered space $\ell^{\infty}_n$ is self-dual.

 On the other hand $\ell^{\infty}_n$ is also an operator space and the normed dual is $\ell^1_n$ via the same identification. The operator space structure on $(\ell^{\infty}_n)^d$ is the operator space $MAX(\ell^1_n) = span \{ u_1,...,u_n \} \subseteq C^*(\bb F_n)$ where $C^*(\bb F_n)$ denotes the full C*-algebra of the free group on $n$ generators and $u_j$ are the generators \cite{Zh}. In this case the norm and cb-norm of the identity map $id: \ell^{\infty}_n \to \ell^1_n$ is $n.$ The cb-condition number is $\|id\|_{cb} \|id^{-1}\|_{cb} =n.$
\end{exam}

\begin{exam} If we consider $M_n$ as an operator system with the usual structure, then \cite{PTT} the map that sends the matrix units $E_{i,j}$ to their dual basis $\{\delta_{i,j} \}$ defines a complete order isomorphism between $M_n$ and $M_n^{d}.$ This map sends the identity operator $I_n = \sum_{j=1}^n E_{j,j}$ to the  trace functional $Tr,$ where $Tr((a_{i,j}) = \sum_{j=1}^n a_{j,j}.$ Thus, $M_n$ is also completely order isomorphic to its dual.

However, recall that the normed dual, with this same identification is the trace class matrices $S_n^1,$ together with their operator space structure. Again the norm, cb-norm, and cb-condition number of the identity map(between these $n^2$ dimensional spaces) is $n.$
\end{exam}

Thus, in both these examples we have operator systems that are completely order isomorphic to their ordered duals, but the identification does not preserve the operator space structure of the dual.

\section{The Operator System SOH(n)}

In this section, for each cardinal number $n,$ we introduce an operator system $SOH(n)$ of dimension $n+1$ based on Pisier's self-dual operator space $OH(n)$ and analyze their properties. In particular, we prove that these operator systems are self-dual as matrix-ordered spaces and that the natural map from $\phi: SOH(n) \to SOH(n)^d$ satisifes $\|\phi \|_{cb} \cdot \|\phi^{-1}\|_{cb} =2,$ which we show is as close to being a complete isometry as is possible for any operator system that is completely order isomorphic to its dual.

We begin with a result that shows that the lower bound of 2 is sharp.

\begin{prop}\label{nonci} Let $\cl S$ be an operator system of dimension at least 2 and assume that $\phi: \cl S \to \cl S^d$ is a complete order isomorphism of $\cl S$ onto its dual space. Then $\|\phi\| \cdot \|\phi^{-1}\| \ge 2.$
\end{prop}
\begin{proof} Let $I$ denote the identity element of $\cl S$ and let $\delta_0 = \phi(I).$ Choose $H=H^* \in \cl S$ that is not in the span of $I.$ Since $\delta_0$ is positive, $\delta_0(H) \in \bb R.$  Replacing $H$ by $H- \delta_0(H)I$ we may assume that $\delta_0(H) =0.$ Now let $\delta_1= \phi(H),$ which is a self-adjoint functional on $\cl S.$  Set $M = \inf \{ r: rI \ge H \}$ and set $m = \sup \{ rI: H \ge rI \}.$ Since $H$ is not a multiple of $I,$ it follows that $m < M.$ For any real numbers $a,b$ we will have that $\|aI + bH \| = max \{ |a+bM|, |a+bm| \}$ and that $aI +bH \ge 0$ iff $min \{ a+bM, a+ bm \} \ge 0.$ Since $\phi$ is a complete order isomorphism, $a \delta_0 + b \delta_1$ is completely positive iff $min \{a+bM, a+bm \} \ge 0.$

Now note that $\|MI - H \| = M-m = \|H - mI\|$ and that $MI -H \ge 0,$ $H- mI \ge 0,$ and so $M \delta_0 - \delta_1$ and $\delta_1 - m \delta_0$ are both completely positive. Let $\delta_1(I) =s.$  The complete positivity of these last two maps, implies that $\|M \delta_0 - \delta_1\| = (M \delta_0 - \delta_1)(I) = M-s \ge 0$ and that
$\| \delta_1 - m \delta_0 \| = (\delta_1 - m \delta_0)(I) = s - m \ge 0.$ Hence, $m \le s \le M.$

Finally,
\[ \|\phi\| \cdot \|\phi^{-1}\| \ge max \{ \frac{\|MI - H\|}{\|M \delta_0 - \delta_1\|}, \frac{\|H - mI\|}{\| \delta_1 - m \delta_0\|} \} = max \{ \frac{M-m}{M-s}, \frac{M-m}{s-m} \} \ge 2.
 \]
This last inequality follows by observing that the minimum of this maximum over $s$ occurs when $s= (M+m)/2.$
\end{proof}

To construct $SOH$, 
we consider the finite dimensional case, the extension to infinite dimensions is standard. We use a few facts that are implicitly contained in Pisier\cite[Exercise~7.2]{Pi2}.
Fix a Hilbert space of dimension $n$ and let $\{ e_i \}$ be an orthonormal basis. Asume that $OH(n) \subseteq B(\cl H)$ is a completely isometric inclusion, so that $e_i$ are identified with operators. Let \[ H_i = \begin{pmatrix} 0 & e_i\\e_i^* & 0 \end{pmatrix} \in B( \cl H \oplus \cl H), \]
so that the $H_i$'s are self-adjoint operators. 

Given matrices, we have that 
\begin{multline*} \| \sum_i A_i \otimes H_i \| = max \{ \| \sum_i A_i \otimes e_i \|, \| \sum_i A_i \otimes e_i^*\| \} =\\
max \{ \| \sum_i A_i \otimes \overline{A_i} \|^{1/2}, \| \sum_i A_i^* \otimes A_i^t \|^{1/2} \} = \| \sum_i A_i \otimes e_i \|. \end{multline*} 
This last equality follows since $A^t \otimes B^t = (A \otimes B)^t$ and so, 
\[ \| \sum_i A_i^* \otimes A_i^t \| = \| (\sum_i \overline{A_i} \otimes A_i)^t \| =
\| \sum_i \overline{A_i} \otimes A_i \| = \| \sum_i A_i \otimes e_i \|^2. \]
Note in particular, we have that $\| \sum_i A_i \otimes e_i \| = \| \sum_i A_i^* \otimes e_i \| = \| \sum_i A_i^t \otimes e_i \|.$

Thus, the map $e_i \to H_i$ is a complete isometry and we have that $OH(n)$ is also the span of these self-adjoint elements. The particular form of these self-adjoint operators will be useful in the sequel.

For notational convenience we let $H_0$ denote the identity operator on $\cl H \oplus \cl H.$

\begin{defn} We let $SOH(n) \subseteq B(\cl H \oplus \cl H)$ denote the $(n+1)$-dimensional operator system that is the span of the set $\{ H_i: 0 \le i \le n \}.$
\end{defn}

We now examine the norm and order structure on $SOH(n).$

\begin{prop}\label{SOHpos} Let $A_i \in M_m, 0 \le i \le n.$ Then the following are equivalent:
\begin{itemize}
\item $\sum_{i=0}^n A_i \otimes H_i$ is positive,
\item $A_0 \otimes H_0 - \sum_{i=1}^n A_i \otimes H_i$ is positive,
\item $A_0 \in M_m^+,$ $A_i=A_i^*, 1 \le i \le n$ and $-A_0 \otimes \overline{A_0} \le \sum_{i=1}^n A_i \otimes \overline{A_i} \le + A_0 \otimes \overline{A_0},$ in $M_m \otimes M_m = M_{m^2}.$
\end{itemize}
\end{prop}
\begin{proof} Let $U = \begin{pmatrix} -I & 0\\0 & I \end{pmatrix} \in B(\cl H \oplus \cl H),$ which is unitary.  Note that $U^*H_0U=H_0$ and $U^*H_iU= - H_i, 1 \le i \le n,$ from which the equivalence of the first two statements follows.

Adding the first two equations shows that $A_0 \ge 0.$ Since a positive element must be self-adjoint it follows that $A_i = A_i^*, 1 \le i \le n.$

To see the final equations, first assume that $A_0$ is positive and invertible. Then $\sum_{i=0}^n A_i \otimes H_i$ is positive iff $(A_0 \otimes H_0)^{-1/2}( \sum_{i=0}^n A_i \otimes H_i)(A_0 \otimes H_0)^{-1/2}$ is positive which is iff $I_m \otimes H_0 + \sum_{i=1}^n B_i \otimes H_i$ is positive, where $B_i = A_0^{-1/2}A_iA_0^{-1/2}.$ As operators on $\cl H \oplus \cl H,$ we have that
\[ \begin{pmatrix} I_{\cl H} & \sum_i B_i \otimes e_i \\ \sum_i B_i \otimes e_i & I_{\cl H} \end{pmatrix} \] is positive.

This last equation is equivalent to requiring that the (1,2)-entry of this operator matrix is a contraction and hence, $\| \sum_i B_i \otimes \overline{B_i} \| \le 1.$  But since these matrices are self-adjoint, this is equivalent to 
\[ - I_m \otimes I_m \le \sum_i B_i \otimes \overline{B_i} \le I_m \otimes I_m. \]
Conjugating this last result by $A_0^{1/2} \otimes \overline{A_0^{1/2}}$ yields the desired inequality.

When $A_0$ is not invertible, one first considers $A_0 + rI_m, r>0$ and then lets $r \to 0.$ This completes the proof. 
\end{proof} 

We now consider the matrix-ordered dual of $SOH(n).$ To this end we let $\delta_i \in SOH(n)^d, 0 \le i \le n$ denote the linear functionals that satisfy, $\delta_i(H_j) = \delta_{i,j}, 0 \le i,j \le n.$

\begin{thm}\label{SOHselfdual} The map $\kappa: SOH(n) \to SOH(n)^d$ defined by $\kappa(H_i) = \delta_i,$  $0 \le i \le n,$ is a complete order isomorphism that satisfies
\[ \| \sum_{i=0}^n A_i \otimes \delta_i \|_{cb} \le \| \sum_{i=0}^n A_i \otimes H_i \| \le 2 \| \sum_{i=0}^n A_i \otimes \delta_i \|_{cb} \] for any matrices $A_0,...,A_n \in M_m$ and any $m$ and $\|\kappa\|_{cb} \cdot \|\kappa^{-1}\|_{cb} =2.$
\end{thm}
\begin{proof} First, we prove that $\kappa$ is completely positive. Keeping the notation from the last proof, assume that $\sum_{i=0}^n A_i \otimes H_i$ is positive.  We must prove that the map $\Phi: SOH(n) \to M_m$ given by $\Phi(X) = \sum_{i=0}^n A_i \otimes \delta_i(X)$ is completely positive. Assume that $A_0$ is invertible and define $B_i$ as above. Let $P= \sum_{i=0}^n P_i \otimes H_i \in M_q(SOH(n))^+.$ We must show that \[ \Phi^{(q)}(P) = \sum_{i=0}^n A_i \otimes P_i \in (M_n \otimes M_q)^+. \] 

Assuming that $P_0$ is also invertible, we set $Q_i = P_0^{-1/2}P_i P_0^{-1/2}.$ By the last Proposition, we have that
$\| \sum_{i=1}^n B_i \otimes e_i \| \le 1$ and $\|\sum_{i=1}^n Q_i \otimes e_i \| \le 1.$
Hence, by the self-duality of $OH(n),$ we have that $\| \sum_{i=1}^n B_i \otimes Q_i \|_{M_m \otimes M_q} \le 1.$ Using the fact that all these matrices are self-adjoint, yields
\[ -I_m \otimes I_q \le \sum_{i=1}^n B_i \otimes Q_i \le +I_m \otimes I_q. \]
Thus, $I_m \otimes I_q + \sum_{i=1}^n B_i \otimes Q_i \ge 0,$ which after conjugation
 by $A_0^{1/2} \otimes P_0^{1/2}$ yields that $\Phi^{(q)}(P) \ge 0.$

Conversely, if $\Phi = \sum_{i=0}^n A_i \otimes \delta_i \in M_m(SOH(n)^d)$ is completely positive, then it follows that $A_0 \ge 0,$ and that $A_i = A_i^*, 1 \le i \le n.$
Taking $B_i$'s as before, we have that $\Psi = I_m \otimes \delta_0 + \sum_{i=1}^n B_i \otimes \delta_i$ is a unital completely positive map and hence completely contractive.  Applying this map to any element $\sum_i 
C_i \otimes e_i \in M_q(OH(n))$ of norm less than one, yields that
$\| \sum_{i=1}^n B_i \otimes C_i \| \le 1.$ Thus, by self-duality of $OH(n)$ we have that $\| \sum_{i=1}^n B_i \otimes \overline{B_i} \| \le 1.$
Hence, $-I_m \otimes I_m \le \sum_{i=1}^n B_i \otimes \overline{B_i} \le + I_m \otimes I_m$ and the Proposition \ref{SOHpos} implies that $\sum_{i=0}^n A_i \otimes H_I$ is positive.

Thus, $\kappa$ is a complete order isomorphism.

We now consider the norm inequalities.  Let $X= \sum_{i=0}^n A_i \otimes H_i,$ set $\Phi = \sum_{i=0}^n A_i \otimes \delta_i$ and assume that $\|X \|_{SOH(n)} \le 1.$  Here, the matrices $A_i$ are no longer necessarily self-adjoint.
We then have that 
\[0 \le  \begin{pmatrix} I_{\cl H}\otimes I_m & X \\X^* & I_{\cl H}\otimes I_m \end{pmatrix} =
\begin{pmatrix} I_m & A_0\\ A_0^* & I_m \end{pmatrix} \otimes H_0 + \sum_{i=1}^n \begin{pmatrix} 0 & A_i\\ A_i^* & 0 \end{pmatrix} \otimes H_i. \]
From the fact that $\kappa$ is completely positive, it follows that
\[  \begin{pmatrix} I_m & A_0\\ A_0^* & I_m \end{pmatrix} \otimes \delta_0 + \sum_{i=1}^n \begin{pmatrix} 0 & A_i\\ A_i^* & 0 \end{pmatrix} \otimes \delta_i=
\begin{pmatrix} I_m \otimes \delta_o & \sum_{i=0}^n A_i \otimes \delta_i \\ \sum_{i=0}^n A_i^* \otimes \delta_i & I_m \otimes \delta_0 \end{pmatrix} = \begin{pmatrix} \Psi & \Phi\\ \Phi^* & \Psi \end{pmatrix}, \] and $\Psi$
is a unital completely positive map. Hence, $\|\Phi\|_{cb} \le 1$ and it follows that $\|\kappa^{(m)}(X)\|_{cb} \le \|X\|$ for any $X \in M_m(SOH(n))$ and any $m.$

Conversely, assume that $\Phi = \sum_{i=0} A_i \otimes \delta_i.$ To prove the other inequality, it will be enough to assume that $\|\Phi\|_{cb} \le 1$ and show that $\| X\|_{SOH(n)} \le 2.$

Since $\|\Phi||_{cb} \le 1,$ there exist unital completely positive maps $\Psi_j: SOH(n) \to M_m, j=1,2$ such that the map $\Gamma = \begin{pmatrix} \Psi_1 & \Phi \\ \Phi^* & \Psi_2 \end{pmatrix}: SOH(n) \to M_{2m}$ is completely positive.  Writing $\Psi_j = \sum_{i=0}^n C^j_i \otimes \delta_i,$ we have that
$\Gamma = \sum_{i=0} \begin{pmatrix} C^1_i & A_i \\A_i^* & C^2_i \end{pmatrix} \otimes \delta_i.$  Moreover, since the maps $\Psi_j$ are unital, $C^1_0 = C^2_0 = I_m.$
By the Proposition and the fact that $\kappa$ is a complete order isomorphism, we know that the fact that $\Gamma$ is completely positive implies that
$\Gamma_1 =\begin{pmatrix} I_m & A_0 \\ A_0^* & I_m \end{pmatrix} \otimes \delta_0 - \sum_{i=1}^n \begin{pmatrix} C^1_i & A_i \\ A_i^* & C^2_i \end{pmatrix} \otimes \delta_i$ is completely positive.  Adding $\Gamma + \Gamma_1,$ and using the positivity, yields that $\|A_0\| \le 1.$

Next, if we let $\Gamma_2$ be the completely positive map that we get by conjugating the coefficients of $\Gamma_1$ by the unitary $U= \begin{pmatrix} -I_m & 0 \\ 0 & I_m \end{pmatrix},$ we find that
$\Gamma_2 = \begin{pmatrix} I_m & -A_0 \\ -A_0^* & I_m \end{pmatrix} \otimes \delta_0 + \sum_{i=1}^n \begin{pmatrix} -C^1_i & A_i \\ A_i^* & -C^2_i \end{pmatrix} \otimes \delta_i.$
The average $1/2(\Gamma + \Gamma_2) = \begin{pmatrix} I_m & 0 \\0 & I_m \end{pmatrix} \otimes \delta_0 + \sum_{i=1}^n \begin{pmatrix} 0 & A_i \\ A_i^* & 0 \end{pmatrix} \otimes \delta_i $ is a unital completely positive map.

Using that $\kappa$ is a complete order isomorphism and replacing the $\delta_i$'s by $H_i$'s, yields that $\|\sum_{i=1}^n A_i \otimes H_i \| \le 1.$
Hence, 
\[\| \sum_{i=0}^n A_i \otimes H_i \| \le \| A_0 \otimes H_0 \| + \| \sum_{i=1}^n A_i \otimes H_i \| \le 2\]
 and the desired inequality follows.

Finally, we have that $\|\kappa\|_{cb} \le 1$ and $\|\kappa^{-1} \|_{cb} \le 2,$ so that $\|\kappa\| \cdot \|\kappa^{-1}\|_{cb} \le 2$ and so we must have equality by Proposition~\ref{nonci}.   
\end{proof}

\begin{remark} By the above results we see that, among all self-dual operator systems, the operator systems $SOH(n)$ acheive the minimal  cb-condition number of 2. However, this does not uniquely characterize these spaces. In fact, $M_2$ is another self-dual operator system that attains this minimum. One other example is $\ell^{\infty}_2,$ but it is not hard to see that this operator system is unitally, completely order isomorphic to $SOH(1).$ It would be interesting to try and characterize the self-dual operator systems that attain this minimal cb-condition number.
\end{remark}
\section{Some Structure Results of SOH(n)}

In \cite{Pi2}, $OH(n)$ is defined in a basis-free fashion. In this section we  show that $SOH(n)$ is also independent of basis, which leads to proving every quotient and operator subsystem of $SOH(n)$ is unitally completely order isomorphic to some $SOH(m)$. We also derive a few properties of $SOH(n)$ that will be useful in the later sections. To avoid ambiguity, whenever we work with $SOH(n)$ and $SOH(m)$, we denote $H_i^{(n)}$ and $H_j^{(m)}$, respectively, the basis elements $H_i$ as given in section 3. 

\begin{prop}\label{SOHcoi}
Let $1 \leq n \leq m$ and let $\{ \vec{u}_i  = (u_{ij}) \in \bb{R}^m \}_{i=1}^n$ be an orthonormal set. Then the map $\Phi \colon SOH(n) \to SOH(m)$ defined by $\Phi(I) = I$ and $\Phi(H_i^{(n)}) := \sum_{j=1}^m u_{ij} H_j^{(m)}$ is a complete order inclusion. 
\end{prop}

\begin{proof}
Consider $\sum_{i=0}^n A_i \otimes H_i^{(n)} \in M_p \otimes SOH(n)$. Let $B_0 = A_0$ and for $ j = 1, \dots, n$, let $B_j = \sum_{i=1}^n u_{ij} A_i$. Then $\sum_{i=0}^n A_i \otimes  \Phi(H_i)$ is 
	\begin{equation*}
			B_0 \otimes H_0^{(m)} + \sum_{i=1}^n A_i \otimes ( \sum_{j=1}^m u_{ij} H_j^{(m)} )	=	B_0 \otimes I + \sum_{j=1}^m B_j \otimes H_j^{(m)}	.
	\end{equation*}
It is obvious that $B_j = B_j^*$; and by orthonormality of the $\vec{u}_i$'s, 
	\begin{align*}
		\sum_{j=1}^m B_j \otimes \overline{B_j} 	&=	\sum_{j=1}^m \left(	\sum_{i, k=1}^n u_{ij} u_{kj} \right)  A_i \otimes \overline{A_k} 	=	\sum_{i, k=1}^n \left(\sum_{j=1}^m u_{ij} u_{kj} \right) A_i \otimes \overline{A_k} \\
							&= \sum_{i, k =1}^n \delta_{i,k} A_i \otimes \overline{A_k}	=	\sum_{i=1}^n A_i \otimes \overline{A_i}.
	\end{align*}
Therefore, $\{A_i\}_{i=0}^n$ satisfies the third condition in Proposition \ref{SOHpos} if and only if $\{B_j\}_{j=0}^m $ satisfies the same condition, proving that $\sum_{i=0}^n A_i \otimes H_i^{(n)} \geq 0$ if and only if $\sum_{i=0}^n A_i \otimes \Phi(H_i^{(n)}) \geq 0$; this  is equivalent to $\Phi$ being a unital complete order inclusion. 
\end{proof}

\begin{cor}
Let $U = [u_{ij}] \in M_n(\bb{R})$ be an orthonormal matrix and set $K_0 = H_0$, $K_i = \sum_{j=1}^n u_{ij} H_j$. Then the map $\Phi \colon SOH(n) \to SOH(n)$ given by $\Phi(H_0) = K_0$ and $\Phi(H_i) = K_i$ is a unital complete order isomorphism. 
\end{cor}

Given $n \leq m$, it is now clear that $SOH(n) \subset_{ucoi} SOH(m)$. We will see that every operator subsystem of $SOH(m)$ is necessarily $SOH(n)$. 

\begin{cor}
If $\cl{T}$ is a operator subsystem of $SOH(m)$ of dimension $n+1$, then $\cl{T}$ is unitally completely order isomorphic to $SOH(n)$.
\end{cor}

\begin{proof}
Let $\{K_0 = I, K_i = K_i^* \colon i =1, \dots, n \}$ be a basis for $\cl{T}$. Without loss of generality, we assume for each $i =1, \dots, n$, $K_i = \sum_{j=1}^m a_{ij} H_j^{(m)}$ for some $a_{ij} \in \bb{R}$. We first claim that the vectors $\vec{a}_i = (a_{ij}) \in \bb{R}^m$ are linearly independent. For if not, then $\vec{a}_i = \sum_{k=1, k \neq i}^n \lambda_k \vec{a}_k$, for some $i$, leading to $K_i = \sum_{j=1}^m \sum_{k=1, k \neq i}^n \lambda_k H_j^{(m)}$, which contradicts our assumption. 

Now consider the $n$-dimensional subspace of $\bb{R}^m$ spanned these $\vec{a_i}$'s. Pick an orthonormal basis $\{\vec{u}_i = (u_{ij}) \in \bb{R}^m\}_{i=1}^n$ for this subspace and define $\Phi \colon SOH(n) \to SOH(m)$  by $\Phi(I) =I$ and $\Phi(H_i^{(n)}) = \sum_{j=1}^m u_{ij} H_j^{(m)}$.  By the last proposition, $\Phi$ is a complete order inclusion. It remains to check that the image of $\Phi$ is $\cl{T}$. Since every $\vec{a_i} = \sum_{k=1}^n \lambda_k^{i} \vec{u}_k$, for each $K_i$ we can write 
	\begin{equation*}
		K_i = \sum_{j=1}^m a_{ij} H_j^{(m)} = \sum_{j=1}^m \sum_{k=1}^n \lambda_k^{i} u_{kj} H_j^{(m)} = \sum_{k=1}^n \lambda_k^{i} \Phi( H_j^{(n)} ),
	\end{equation*}
proving that $\Phi( SOH(n) ) = \cl{T}$. Consequently $\cl{T} \cong_{ucoi} SOH(n)$ via $\Phi$.
\end{proof}

Hence every operator subsystem of $SOH(n)$ is again of the same form. The next result characterizes quotients of $SOH(n)$ based on self-duality. 

\begin{prop}
Let $\cl{J}$ be a non-trivial self-adjoint subspace of $SOH(n)$. Then the following are equivalent:
	\begin{enumerate}
		\item 	$\cl J$ is the kernel of some unital, completely positive map with domain $SOH(n)$.
		\item 	There exist $m < n$ and a surjective unital completely positive map $\phi \colon SOH(n) \to SOH(m)$ such that $\cl J = \ker (\phi )$.
		\item 	There is unital completely positive map $\phi$ on $SOH(n)$ for which $\cl J = \ker(\phi)$.
	\end{enumerate}
\end{prop}

\begin{proof}
The direction $(2) \Longrightarrow (3) \Longrightarrow (1)$ is obvious. Now assume (1) and let $q \colon SOH(n) \to SOH(n)/\cl  J$ be the canonical quotient map. Then $q^d \colon ( SOH(n)/\cl J )^d \to SOH(n)^d = SOH(n)$ is a unital complete order embedding \cite{FP}. Since $\cl J$ is non-trivial, $(SOH(n)/\cl J)^d$ has dimension $m < n$ and by the last corollary $(SOH(n)/\cl J)^d \cong SOH(m)$. By duality, $SOH(n)/ \cl J \cong SOH(m)^d = SOH(m)$. 
\end{proof}

In \cite[Section 8]{Ka}, it is shown that the coproduct of two operator systems $\cl{S}$ and $\cl{T}$ can be obtained by operator system quotients. Namely, $\cl{S} \oplus_1 \cl{T} \cong_{ucoi} ( \cl{S} \oplus \cl{T} ) / \cl J$, where $\cl J = \bb{C}(1_{\cl{S}}, -1_{\cl{T}})$. 

\begin{prop}
For any $p \in \bb{N}$, let $H_0^{(p)},...,H_p^{(p)}$ denote the canonical basis for $SOH(p).$  Then for any $n,m \in \bb N$, the map $\phi: SOH(n) \oplus SOH(m) \to SOH(n + m )$ defined by $\phi(H_j^{(n)}) = H_j^{(n+m)}, 0 \le j \le n$ and $\phi(H_j^{(m)}) = \begin{cases} H_0^{(n+m)}, & j=0\\ H_{n+j}^{(n+m)}, & j > 0 \end{cases}$ induces a unital completely positive map $\Phi: SOH(n) \oplus_1 SOH(m) \to SOH(n+m),$ but this map is not an order isomorphism.
\end{prop}
\begin{proof} It is easily checked that the restriction of $\phi$ to each direct summand is a unital completely positive map. Hence, $\Phi$ is a unital completely positive map by the universal properties of the coproduct.

To see that $\Phi$ is not an order isomorphism,
it suffices to show that $SOH(1) \oplus_1 SOH(1) \neq SOH(2)$. Suppose the contrary and consider the positive element $P = \sqrt{2} H_0^{(2)} + H_1^{(2)} + H_2^{(2)}$ in $SOH(2)$. Then there must be positive numbers $a$ and $b$ such that $(a H_0^{(1)} + H_1^{(1)})$ and $(b H_0^{(1)} + H_1^{(1)})$ are positive in $SOH(1)$ and sum to $P$ in $SOH(2)$. By Proposition \ref{SOHpos}, each of  $a^2$ and $b^2$ is greater than $1$; however $a + b = \sqrt{2}$ implies that $2ab \leq 0 $, contradicting $a$ and $b$ are positive. 
\end{proof}

\begin{remark} In an earlier version of this paper, we erroneously claimed that $\Phi$ was a complete order isomorphism. We would like to thank Ali S. Kavruk for pointing out this error.
\end{remark}
\begin{prop}
Let $\cl{S}$ be an operator system and $\{h_i \colon h_i = h_i^*, ||h_i|| \leq 1 \}_{i=1}^n \subset \cl{S}$. Then there is $r > 0$ such that the map $\phi \colon SOH(n) \to \cl{S}$ given by $H_0 \mapsto r 1_{\cl{S}}$, $H_i \mapsto h_i$ is completely positive. 
\end{prop}

\begin{proof}
Choose $r > n^{1/2}$ and suppose $A_0 \otimes H_0 + \sum_{i=1}^n A_i \otimes H_i$ is positive in $M_m \otimes SOH(n)$. We will show that $r A_0 \otimes 1_{\cl{S}} + \sum_{i=1}^n A_i \otimes h_i$ is positive. First assume $A_0 > 0$ is invertible. We claim 
	\begin{equation*}
		\begin{bmatrix}
		r A_0 \otimes 1_{\cl{S}}	&	\sum_{i=1}^n A_i \otimes h_i	\\
		\sum_{i=1}^n A_i^* \otimes h_i^*	&	r A_0 \otimes 1_{\cl{S}}
		\end{bmatrix}
	\end{equation*}
is positive in $M_{2m} \otimes SOH(n)$, which is equivalent to
	\begin{equation*}
		r^{-1} || \sum_{i=1}^n A_0^{-1/2} A_i A_0^{-1/2} \otimes h_i ||_{M_m \otimes \cl{S}} \leq 1. 
	\end{equation*}
Write $B_i = A_0^{-1/2} A_i A_0^{-1/2}$, then by Proposition \ref{SOHpos}, $|| \sum_{i=1}^n B_i \otimes \overline{B_i} || \leq 1$. Now embed $\cl{S} \subset B( \cl{H} )$ and regard $h_i \otimes \overline{h_i}$ as an operator in $B( \cl{H} \otimes \overline{H} )$. Then by a version of Cauchy-Schwarz inequality due to Haagerup \cite[Lemma 2.4]{Ha},   
	\begin{align*}
			r^{-1} || \sum_{i=1}^n B_i \otimes h_i ||_{M_m \otimes \cl{S}} 	&\leq 	r^{-1} || \sum_{i=1}^n B_i \otimes \overline{B_i} ||^{1/2}_{M_{2m}} \cdot || \sum_{i=1}^n h_i \otimes \overline{h_i} ||^{1/2}_{B(\cl{H} \otimes \overline{\cl{H} })}	\\
																			& \leq r^{-1} n^{1/2} < 1.
	\end{align*}
Hence, the above matrix is positive as claimed. Pre and post multiplying it by $[1, 1]$ shows that $2 (r A_0 \otimes 1_{\cl{S}} + \sum_{i=1}^n A_i \otimes h_i )$ is positive.
When $A_0$ is not invertible, apply the standard $A_0 + \varepsilon I_m$ argument as in the proof of Proposition \ref{SOHpos}. Consequently, $\phi$ is completely positive. 
\end{proof}

\begin{cor}
In the previous settings, if $\cl{S}$ is an operator system, then the map $\theta \colon \cl{S}^d \to SOH(n)$ by $\theta(f) = r f(1_{\cl{S}}) H_0 + \sum_{i=1}^n f(h_i) H_i$ is completely positive.
\end{cor}

\begin{proof}
The dual map $\phi^d \colon \cl{S}^d \to SOH(n)^d$, $\phi^d (f) (H_i) = f \circ \phi (H_i)$, is completely positive. Let $\kappa \colon SOH(n) \to SOH(n)^d$ be the map $h_i \mapsto \delta_i$ as in Theorem \ref{SOHselfdual}. Then by self-duality of $SOH(n)$, the map $\kappa^{-1} \circ \phi^d \colon \cl{S}^d \to SOH(n)$ is completely positive and an easy calculation shows that $\kappa^{-1} \circ \phi^d (f) = \theta(f)$.
\end{proof}

\section{The $\gamma_{soh}$-Tensor Product} 
One of the important Banach space tensor products arises via factorization of bounded maps through Hilbert space. In this section and the next we construct a tensor product of two operator systems that arises from factorization of completely positive maps through SOH.

In \cite{Ng}, it is shown that the positive cone of the maximal tensor
product of finite dimensional operator systems, $\cl{S} \otimes_{\max}
\cl{T}$, can be identified with the completely positive maps from
$\cl{S}^d$ to $\cl{T}$ that factor through $M_n$ approximately;
equivalently these are the nuclear maps. Motivated by this
characterization, we will construct the $\gamma_{soh}$ tensor product
similarly by using $M_p(SOH(n))$ instead of $M_n$. We show that
$\phi_1 \otimes \phi_2 \colon \cl{S}_1 \otimes_{\gamma_{soh}} \cl{S}_2
\to \cl{T}_1 \otimes_{\gamma_{soh}} \cl{T}_2$ is completely positive
whenever $\phi_i \colon \cl{S}_i \to \cl{T}_i$ is completely positive.
We prove that $\gamma_{soh}$ is a functorial and symmetric tensor
product structure in the category of  finite dimensional operator
systems. We also prove that $\gamma_{soh}$ is a distinct tensor
product from many of the functorial tensors studied in \cite{KPTT1, KPTT2, FKPT}. 

\begin{defn} Let $\cl S$ and $\cl T$ be operator systems.
We say that $\hat{u}: \cl S^d \to \cl T$ {\bf factors through $SOH$ approximately}, provided there exist nets of completely positive maps $\phi_{\lambda} \colon \cl{S}^d \to M_{p_{\lambda}} ( SOH(n_{\lambda}) )$ and $\psi_{\lambda} \colon M_{p_{\lambda}} (SOH(n_{\lambda})) \to \cl{T}$ such that $\psi_{\lambda} \circ \phi_{\lambda}$ converges to $\hat{u}$ in the point-norm topology. 
\end{defn}

\begin{defn}[The $\gamma_{soh}$-cone]
Let $\cl{S}$ and $\cl{T}$ be finite dimensional operator systems. Define
	\begin{equation*}
		\cl{C}_1^{\gamma} (\cl{S}, \cl{T} ) 	:= \{ u \in \cl{S} \otimes \cl{T} \colon \hat{u} \text{ factors through } SOH \text{ approximately} \}.
	\end{equation*}
For $u = [u_{ij}] \in M_n( \cl{S} \otimes \cl{T} )$, we regard $\hat{u} = [ \hat{u_{ij}} ]$ as a map from $\cl{S}^d$ to $M_n(\cl{T})$. Thus there is no confusion to define $\cl{C}_n^{\gamma}( \cl{S}, \cl{T} ) = \cl{C}_1^{\gamma}( \cl{S}, M_n( \cl{T} ) )$ in $M_n( \cl{S} \otimes \cl{T} )$. We denote the triple $( \cl{S} \otimes \cl{T},  \{ \cl{C}_n^{\gamma}( \cl{S}, \cl{T} ) \}_{n=1}^{\infty}, 1_{\cl{S}} \otimes 1_{\cl{T}} )$ by $\cl{S} \otimes_{\gamma_{soh}} \cl{T}$. 
\end{defn}

\begin{prop}
The collection $\{ \cl{C}_n^{\gamma}( \cl{S}, \cl{T} ) \}$ is a compatible family of proper cones of $\cl{S} \otimes \cl{T}$.
\end{prop}

\begin{proof}
Since $\cl{C}_n^{\gamma}( \cl{S}, \cl{T} ) = \cl{C}_1^{\gamma}( \cl{S}, M_n(\cl{T}))$, it suffices to check that $\cl{C}_1^{\gamma}( \cl{S}, \cl{T} )$ is a proper cone. 
It is obvious that $\cl{C}_1^{\gamma}(\cl{S}, \cl{T} )$ is closed under positive scalar multiplication. Let $u_1, u_2 \in \cl{C}_{\gamma}^{1}( \cl{S}, \cl{T} )$, so there are nets of completely positive maps $\phi_{\lambda_k}, \psi_{\lambda_k}$, where $k= 1, 2$ such that $\lim_{\lambda} \psi_{\lambda_k} \circ \phi_{\lambda_k} =  \hat{u_k}$ in the point-norm topology.

Consider the directed set $\Lambda = \{ (\lambda_1, \lambda_2) \}$ with the natural ordering. For each $\lambda = ( \lambda_1, \lambda_2 ) \in \Lambda$, regard $M_{p_{\lambda}} = M_{p_{\lambda_1}} \oplus M_{p_{\lambda_2}}$ as the 2-by-2 block and let $n_{\lambda} = \max\{n_{\lambda_1}, n_{\lambda_2} \}$. Note that every completely positive map on $SOH(n_{\lambda_k})$, $k=1, 2$, can be extended naturally on $SOH(n_{\lambda})$. Thus without loss of generality we may assume that $\phi_{\lambda_k}$ maps into $M_{p_{\lambda}} \otimes SOH(n_{\lambda})$ and $\psi_{\lambda_k}$ has domain $M_{p_{\lambda}} \otimes SOH(n_{\lambda_k})$. 

Thus, for each $\lambda = (\lambda_1, \lambda_2)$, we take $M_{p_\lambda} (SOH( n_{\lambda} ) )$, with completely positive maps $\phi_{\lambda} = \phi_{\lambda_1} \oplus \phi_{\lambda_2}$ and $\psi_{\lambda} = \psi_{\lambda_1} \oplus \psi_{\lambda_2}$. It remains to check that $\psi_{\lambda} \circ \phi_{\lambda}$ converges to $\widehat{(u_1 + u_2)}$ in the point-norm topology. Indeed, given $f \in \cl{S}^d$ and $\varepsilon > 0$, by assumption there exist $\mu_1$ and $\mu_2$ so that $|| \hat{u_k} (f) - \psi_{\lambda_k} \circ \phi_{\lambda_k} (f) || < \frac{\varepsilon}{2}$, for $\lambda_k > \mu_k$. Thus if $\mu = (\mu_1, \mu_2)$ and $\lambda > \mu$, then 
	\begin{equation*}
		|| \widehat{(u_1+u_2)} (f) - (\psi_{\lambda} \circ \phi_{\lambda} ) (f) || \leq	\sum_{k=1}^2 || \hat{u_k} (f) - (\psi_k \circ \phi_k)( f) || < \varepsilon
	\end{equation*}
shows that  $u_1 + u_2$ is in $\cl{C}_1^{\gamma}( \cl{S}, \cl{T} )$. 

Next we verify compatability. Let $u = [u_{ij}] \in \cl{C}_n^{\gamma}( \cl{S}, \cl{T} )$ with $\hat{u}$ factors through $SOH$ approximately via nets $\psi_{\lambda}$ and $\phi_{\lambda}$. Write $A = [a_{kl}] \in M_{m, n}$, and write $w = A u A^* \in M_m( \cl{S} \otimes \cl{T} )$. We claim that $\hat{w}$ also factors through $SOH$ approximately via the nets $(\theta_A \circ \psi_{\lambda} )$ and $\phi_{\lambda}$, where $\theta_A \colon M_n(\cl{T}) \to M_m( \cl{T} )$ by $B \mapsto ABA^*$ is completely positive. To this end, note that by writing $w = [ \sum_{k, l}^n a_{i, k} u_{k, l} \overline{a_{l, j}} ]_{i, j=1}^m$, for each $f \in \cl{S}^d$
	\begin{align*}
		\hat{w}(f) 	&=	\left[		\sum_{k,l=1}^n	\hat{( a_{i,k} u_{k,l} \overline{a_{l,j}} )} (f) 	\right]_{i,j=1}^m	\\
							&=	\left[		\sum_{k,l=1}^n	a_{i,k} \hat{u}_{k,l}(f)  \overline{a_{l,j}}	\right]_{i,j=1}^m	
							=	A \hat{u}(f) A^*	=	( \theta_A \circ \hat{u} ) (f). 
	\end{align*}
Thus, for each $f \in \cl{S}^d$, 
	\begin{equation*}
		|| \hat{w} (f) - \theta_A \circ \psi_{\lambda} \circ \phi_{\lambda} (f) || 	=	|| \theta_A \circ ( \hat{u} - \psi_{\lambda} \circ \phi_{\lambda} ) (f)   || \to 0.
	\end{equation*}
Therefore, $\{ \cl{C}_n^{\gamma} (\cl{S}, \cl{T} ) \}$ is a compatible family of proper cones. 
\end{proof}

\begin{prop}
The unit $1_{\cl{S}} \otimes 1_{\cl{T}}$ is an Archimedean matrix order unit for $\cl{S} \otimes_{\gamma_{soh}} \cl{T}$. 
\end{prop}

\begin{proof}
Again by identifying $\cl{C}_n^{\gamma} (\cl{S}, \cl{T} ) = \cl{C}_1^{\gamma} ( \cl{S}, M_n(\cl{T} ) )$, it suffices to prove that $1_{\cl{S}} \otimes 1_{\cl{T}}$ is an Archimedean order unit for $\cl{S} \otimes_{\gamma} \cl{T}$ on the ground level. Let $u \in \cl{S} \otimes \cl{T}$ be self-adjoint, we must find an $r > 0$ so that $r 1_{\cl{S}} \otimes 1_{\cl{T}} - u$ is in $\cl{C}_{\gamma}^1(\cl{S}, \cl{T})$. Withoutloss of generality, we may assume $u = \sum_{i=1}^{n} x_i \otimes y_i$, where $x_i = x_i^*$ and $y_i = y_i^*$. By Proposition 4.6 and Corollary 4.7, there exist $r_1, r_2 > 0$ such that the map $\phi \colon \cl{S}^d \to SOH(n)$ by $\phi(f) = r_1 f(1_{\cl{S}}) H_0 - \sum_{i=1}^n f( x_i ) H_i$, and $\psi \colon SOH(n) \to \cl{T}$ by $\psi(H_0) = r_2 1_{\cl{T}}$, $\psi(H_i) = y_i$ are completely positive. Choose $r = r_1 r_2$, then
	\begin{align*}
		\psi ( \phi (f) ) 	&=	r_1 r_2 f(1_{\cl{S}}) 1_{\cl{T}} - \sum_{i=1}^n f(x_i)y_i = \widehat{ (r 1_{\cl{S}} \otimes 1_{\cl{T}}  - u)} (f)
	\end{align*} 
shows that $(\widehat{r1_{\cl{S}} \otimes 1_{\cl{T}} - u} )$ factors through $SOH(n)$ exactly. Consequently, $1_{\cl{S}} \otimes 1_{\cl{T}}$ is an order unit for $\cl{S} \otimes_{\gamma} \cl{T}$. 

Finally suppose $u = \sum_{i=0}^n x_i \otimes y_i \in  \cl{S} \otimes \cl{T} $ and for each $\varepsilon > 0$, $u_{\varepsilon} = u + \varepsilon (1_{\cl{S}} \otimes 1_{\cl{T}} ) \in \cl{C}_1^{\gamma}(\cl{S}, \cl{T} )$. For each $\varepsilon$, there is a net of completely positive maps $\phi_{\lambda_{\varepsilon}}$ and $\psi_{\lambda_{\varepsilon}}$ such that
	\begin{equation*}
		\xymatrix{
			\cl{S}^d 	\ar[rr]^{\hat{u_{\varepsilon}}}	\ar[rd]_{\phi_{\lambda_{\varepsilon}}} &			&		\cl{T}		\\
					&	M_{p_{\lambda_{\varepsilon}}} ( SOH(n_{\lambda_{\varepsilon}} ) ) 	\ar[ru]_{\psi_{\lambda_{\varepsilon}}}		&
		}
	\end{equation*}
and $|| \hat{u_{\varepsilon}}(f) - ( \psi_{\lambda_{\varepsilon}} \circ \phi_{\lambda_{\varepsilon}} ) (f) || \to 0$, for each $f \in \cl{S}^d$. 

Hence for each fixed $\varepsilon$, by finite dimensionality of $\cl{S}^d$, there exist a sufficiently large $k > \frac{1}{\varepsilon}$ and a pair of completely positive maps $\phi_{\lambda_{(\varepsilon, k)}}$ and $\psi_{\lambda_{(\varepsilon, k ) } }$ from the net $(\psi_{\lambda_{\varepsilon}} \circ \phi_{\lambda_{\varepsilon}} )$, such that $|| \hat{u_{\varepsilon}}(f) - ( \psi_{\lambda_{(\varepsilon, k)}} \circ \phi_{\lambda_{(\varepsilon, k)}} ) (f) || < \frac{1}{k}$, for every $|| f || \leq 1$.

Consider the directed set $\Lambda$ consisting of $(\varepsilon, k)$ subject to the above condition, and order it by $(\varepsilon, k) \leq ( \varepsilon', k')$ if and only if $\varepsilon' \leq \varepsilon$ and $k' \geq k$. Now we claim that $(\psi_{\lambda} \circ \phi_{\lambda} )_{\lambda \in \Lambda}$ converges to $\hat{u}$ in the point-norm topology. Given $f \in \cl{S}^d$ with $|| f || \leq 1$, for each $m > 0$, consider for $\varepsilon > \frac{1}{2m}$ and those $\lambda = ( \varepsilon, k)$, 
	\begin{align*}
		|| \hat{u} (f) - ( \psi_{\lambda} \circ \phi_{\lambda} ) (f) || 	&=	|| \hat{u} (f) - \hat{u_{\varepsilon}} (f) + \hat{u_{\varepsilon}} (f) - ( \psi_{\lambda} \circ \phi_{\lambda} ) (f) 	||	\\
				&\leq 	||  \hat{u} (f) - \hat{u_{\varepsilon}} (f)  || + ||	 \hat{u_{\varepsilon}} (f) - ( \psi_{\lambda} \circ \phi_{\lambda} ) (f)	 ||	\\
				&< 	\frac{1}{2m}   + \frac{1}{2m}.
	\end{align*}
Therefore, $\hat{u}$ factors through $M_p(SOH(n))$ approximately and $u \in \cl{C}_1^{\gamma}( \cl{S}, \cl{T} )$. Consequently, $1_{\cl{S}} \otimes 1_{\cl{T}}$ is an Archimedean matrix order unit. 
\end{proof}

\begin{defn}
The triple $ (\cl{S} \otimes \cl{T}, \cl{C}_n^{\gamma}(\cl{S}, \cl{T}), 1_{\cl{S}} \otimes 1_{\cl{T}} )$ is an operator system, and we denote it by $\cl{S} \otimes_{\gamma_{soh}} \cl{T}$. 
\end{defn}

\begin{thm}
The $\gamma_{soh}$-tensor defines a functorial operator system tensor product structure in the category of finite dimensional operator systems.
\end{thm}

\begin{proof}
Let $P \in M_n( \cl{S} )^+$ and $Q \in M_m( \cl{T} )^+$. Note that by regarding $\cl{S} = \cl{S}^{dd}$ and $P \colon \cl{S}^d \to M_n$, then $\hat{(P \otimes Q)} \colon \cl{S}^d \to M_{nm}(\cl{T})$ maps $f$ to $P(f) \otimes Q$. Moreover,  $\hat{P \otimes Q}$ factors through $M_n \otimes SOH(1)$ via
\begin{equation*}
 \xymatrix{    \cl{S}^d	\ar[rd]_{P \otimes H_0 } \ar[rr]^{\hat{(P \otimes Q)}}	&		&	M_{nm}(\cl{T})	\\
	&M_n \otimes SOH(1)	\ar[ru]_{ I_n \otimes Q}	& }
\end{equation*}
Therefore, $P \otimes Q \in \cl{C}_{nm}^{\gamma}(\cl{S}, \cl{T})$. 

For the functorial property, let $\rho \colon \cl{S} \to \cl{V}$ and $\kappa \colon \cl{T} \to \cl{W}$ be completely positive maps between finite dimensional operator systems, and let $u \in \cl{S} \otimes_{\gamma} \cl{T}$ be positive. Thus $\hat{u}$ factors through $M_p ( SOH(n) )$ approximately via some $\phi_{\lambda}$ and $\psi_{\lambda}$. Let $w = ( \rho \otimes \kappa )(u) \in \cl{V} \otimes \cl{W}$. Notice this diagram 
	\begin{equation*}
		\xymatrix{
			\cl{V}^d		\ar[rr]^{\hat{w}} \ar[d]_{\rho^d} &			&	\cl{W}	\\
			\cl{S}^d	\ar[rr]^{\hat{u}} \ar[dr]_{\phi_{\lambda}}&		&	\cl{T}	\ar[u]_{\kappa} \\
							&		M_{p_{\lambda}}(SOH(n_{\lambda}))		\ar[ru]_{\psi_{\lambda}} &
		}
	\end{equation*}
commutes and the maps are all completely positive. Indeed, if $w = \sum_{i=1}^n \rho( x_i ) \otimes \kappa (y_i)$, where $u = \sum_{i=1}^n x_i \otimes y_i$, then for each $f \in \cl{V}^d$,  
	\begin{align*}
		\hat{w}(f) 	&=	\sum_{i=0}^n	f ( \rho (x_i) ) \kappa (y_i) 	=	( \kappa \circ \hat{u} \circ \rho^d ) (f)	\\	
							&=	\lim_{\lambda} ( \kappa \circ \psi_{\lambda}) \circ ( \phi_{\lambda} \circ \rho^d ) (f) .					
	\end{align*}
Therefore, $\hat{w}$ also factors through $M_p ( SOH(n)) $ approximately and $w \in  (\cl{V} \otimes_{\gamma_{soh}} \cl{W})^+$. For $u = [u_{ij}] \in M_n ( \cl{S} \otimes_{\gamma_{soh}} \cl{T} )^+$, in the same vein we regard $\hat{u} \colon \cl{S} \to M_n( \cl{T} )$. Then by replacing $\kappa$ by $\kappa \otimes I_n$ and $\cl{W}$ by $M_n( \cl{W} )$ we deduce that $\hat{u}$ factors through $SOH$ approximately. Consequently $\rho \otimes \kappa$ is completely positive and the $\gamma_{soh}$-tensor product is functorial.
\end{proof}

\begin{remark}
In \cite{FKPT}, the $ess$-tensor product $\cl{S} \otimes_{ess} \cl{T}$ arises by the inclusion in $C_e^*(\cl{S}) \otimes_{\max} C_e^* (\cl{T} )$, where $C_e^*(\cl{S} )$ is the enveloping $C^*$-algebra of $\cl{S}$. It was yet to know whether this tensor product is functorial. Recently in \cite[Proposition 3.2]{GL}, it is shown that the $ess$-tensor product is not functorial. This allows us to distinguish $\gamma_{soh}$ from $ess$.
\end{remark}

\begin{cor}
The $\gamma{soh}$-tensor product is not the $ess$-tensor product.
\end{cor}

We deduce further properties of the $\gamma_{soh}$-tensor product. 

\begin{prop}
The $\gamma_{soh}$-tensor is symmetric.
\end{prop}

\begin{proof}
If $u \in ( \cl{S} \otimes_{\gamma_{soh}} \cl{T} )^+$, then by self-duality of $SOH(n)$ we see that
	\begin{equation*}
		\xymatrix{			
				\cl{T}^d		\ar[rr]^{ \hat{u}^d } 		\ar[rd]_{\psi_{\lambda}^d}	&			&		\cl{S}^{dd} = \cl{S}	\\
								&	M_{p_{\lambda}} ( SOH(n_{\lambda}) )\ar[ru]_{\phi_{\lambda}^d} 	&
		}
	\end{equation*}
commutes. Indeed, if $u = \sum_{i=1}^n x_i \otimes y_i$, then for $g \in \cl{T}^d$ and $f \in \cl{S}^d$,
	\begin{align*}
		( \hat{u}^d )(g) (f) 	&=	g ( \hat{u} (f) ) = \sum_{i=1}^n g(y_i) f(x_i) = \hat{u}(f)(g)
	\end{align*}
Hence $\hat{u}$ admits an approximate factorization through $M_p(SOH(n))$ if and only if $\hat{u}^d$ does. At the matrix level, we identify $M_n( \cl{S} \otimes_{\gamma_{soh}} \cl{T} )^+ = (\cl{S} \otimes_{\gamma_{soh}}  M_n(\cl{T} ) )^+ = (M_n(\cl{T})  \otimes_{\gamma_{soh}} \cl{S} )^+ = (\cl{T} \otimes_{\gamma_{soh}} M_n(\cl{S}) )^+ = M_n( \cl{T} \otimes_{\gamma_{soh}} \cl{S} )^+$. This shows that $x \otimes y \mapsto y \otimes x$ is a complete order isomorphism from $\cl{S} \otimes_{\gamma_{soh}} \cl{T}$ onto $\cl{T} \otimes_{\gamma_{soh}} \cl{S}$. 
\end{proof}

In \cite{KPTT1}, there are some tensor products constructed using the injective envelope. These come from the identifications, $\cl{S} \otimes_{el} \cl{T} \subset_{coi} \cl{I}( \cl{S} ) \otimes_{\min} \cl{T}$, where $\cl{I}(\cl{S})$ is the injective envelope of $\cl{S}$, and likewise for $\cl{S} \otimes_{er} \cl{T}$. It turns out that the $el$ and $er$-tensor products are not symmetric. 

\begin{cor}
The $\gamma_{soh}$-tensor product is neither the $er$ nor the $el$-tensor product.
\end{cor}

\begin{thm}
The $\gamma_{soh}$-tensor product is not the maximal tensor product. In particular, for $n \geq 2$, $SOH(n) \otimes_{\gamma_{soh}} SOH(n) \neq SOH(n) \otimes_{\max} SOH(n)$.
\end{thm}

\begin{proof}
By self-duality of $SOH(n)$, it suffices to show that $SOH(n)^d \otimes_{\gamma_{soh}} SOH(n) \neq SOH(n)^d \otimes_{\max} SOH(n)$. Consider the element $u = \sum_{i=0}^n \delta_i \otimes H_i$. Note that $\hat{u}$ is in fact the identity map on $SOH(n)$ and factors through $SOH$ trivially, so $u \in (SOH(n)^d \otimes_{\gamma_{soh}} SOH(n))^+$. 

On the other hand, if $u \in SOH(n)^d \otimes_{\max} SOH(n)$ were positive, then by \cite[Theorem 16]{Ng}, $\hat{u}$ factors through the matrix algebras approximately. By a result of \cite[Corollary 3.2]{HP}, $SOH(n)$ must be (min, max)-nuclear and thus is unitally completely order isomorphic to a finite dimensional C*-algebra. However it follows that $OH(n)$ could be completely isometrically represented on a finite dimensional Hilbert space and is hence 1-exact, contradicting Pisier's result \cite{Pi2}. Therefore, $u$ is not positive in $SOH(n)^d \otimes_{\max} SOH(n)$. Consequently the two operator systems are not completely isomorphic. 
\end{proof}

We have seen that $\gamma_{soh}$ is indeed a new tensor product. The next natural question is to ask which operator systems are nuclear with respect to $\gamma_{soh}$. The following result characterizes $(\min, \gamma_{soh})$-nuclearity by identifying the matricial cone structures of the minimal tensor product to completely positive maps. 

\begin{thm}
Let $\cl{S}$ and $\cl{T}$ be finite dimensional operator systems. Then $\cl{S} \otimes_{\min} \cl{T} = \cl{S} \otimes_{\gamma_{soh}} \cl{T}$ if and only if every completely positive map from $\cl{S}^d$ to $\cl{T}$ factors through $SOH$ approximately. 
\end{thm}

\begin{proof}
By \cite[Proposition 1.9]{FP}, $\cl{S} \otimes_{\min} \cl{T} =_{ucoi}
( \cl{S}^d \otimes_{\max} \cl{T}^d )^d$, whose cone $(\cl{S}^d
\otimes_{\max} \cl{T}^d)^{d, +}$ is in one-to-one corespondence to
$CP( \cl{S}^d, \cl{T})$.  Hence $\phi \in CP( \cl{S}^d, \cl{T} )$ if
and only if $\phi = \hat{u}$ for some $u \in (\cl{S} \otimes_{\min}
\cl{T})^+$; and $\hat{u}$ factors through $SOH$ approximately
if and only if $u \in (\cl{S} \otimes_{\gamma_{soh}}
\cl{T})^+$. Consequently, $(\cl{S} \otimes_{\min} \cl{T})^+ = ( \cl{S}
\otimes_{\gamma_{soh}} \cl{T})^+$ if and only if every completely
positive $\phi \colon \cl{S}^d \to \cl{T}$ admits such a factorization. At the matrix level, we identify $M_n( \cl{S} \otimes_{\tau} \cl{T} )^+$ to $(\cl{S} \otimes_{\tau} M_n(\cl{T} ) )^+$ for $\tau = \min, \gamma_{soh}$; then the result follows from the base case.
\end{proof}

\begin{cor}
$SOH(n)$ is $(\min, \gamma_{soh})$-nuclear. 
\end{cor}

\begin{cor}
The $\gamma_{soh}$-tensor product is not self-dual.
\end{cor}

\begin{proof}
Suppose $\gamma_{soh}$ is self-dual; that is, $(\cl{S} \otimes_{\gamma_{soh}} \cl{T} )^d = \cl{S}^d \otimes_{\gamma_{soh}} \cl{T}^d$.  Then $SOH(n) \otimes_{\min} SOH(n)^d = SOH(n) \otimes_{\gamma_{soh}} SOH(n)^d$ and by dualizing one obtains $SOH(n)^d \otimes_{\max} SOH(n) = SOH(n)^d \otimes_{\gamma_{soh}} SOH(n)$, which is a contradiction. 
\end{proof}

\section{Extension to Infinite Dimensional Operator Systems}
In this section we show that every functorial tensor product structure defined on the category of finite dimensional operator systems can be extended to infinite dimensional operator systems. We also prove that this extension preserves symmetry, injectivity, and projectivity. Therefore, the $\gamma_{soh}$-tensor product defined in the previous section can now be extended to infinite dimensional operator systems. 

Given an operator system $\cl{S}$, we denote the collection of finite dimensional operator subsystems of $\cl{S}$ by $\cl{F}( \cl{S} )$. 

\begin{defn}
Let $\tau$ be a functorial tensor product structure on the category of finite dimensional operator systems. We define $\tilde{\tau}$ on the category of operator systems in the following way: Given $\cl{S}$ and $\cl{T}$, for each $n \in \bb{N}$, define the family of proper cones
\begin{equation*}
		\cl{C}_n^{\tilde{\tau}} ( \cl{S},  \cl{T} )	:=	\bigcup_{E \in \cl{F}(\cl{S}), F \in \cl{F}(\cl{T})} M_n( E \otimes_{\tau} F)^+.
	\end{equation*}
\end{defn}

\begin{thm}
$\tilde{\tau}$ defines a functorial tensor product structure on the category of operator systems. 
\end{thm}

\begin{proof}
Let us denote $\cl{C}_n^{\tilde{\tau}} = \cl{C}_n^{\tilde{\tau}}( \cl{S}, \cl{T} )$.  We first claim that it defines a matrix-ordering on $\cl{S} \otimes \cl{T}$. It is trivial that $\cl{C}_n^{\tilde{\tau}}$ is a proper cone for each $n$. To show that this is a matrix-ordering, we first check that for each $m, n \in \bb{N}$, $A \in M_{n, m}( \bb{C} )$, $A^* \cl{C}_n^{\tilde{\tau}} A \subset C_m^{\tilde{\tau}}$. Since every $B \in \cl{C}_n^{\tilde{\tau}}$ belongs to $M_n( E \otimes_{\tau} F )^+$, for some $E \in \cl{F}(\cl{S})$ and $F \in \cl{F}( \cl{T} )$, we have $A^* B A \in \cl{C}_m( E \otimes_{\alpha} F ) \subset \cl{C}_m^{\tilde{\tau}}$. 		

To see that $1 \otimes 1$ is an Archimedean matrix order unit for $(\cl{S} \otimes \cl{T}, \cl{C}_n^{\tilde{\tau}} )$, consider $A \in M_n ( \cl{S} \otimes \cl{T} )$ such that for each $\varepsilon > 0$, $A_{\varepsilon} = \varepsilon ( 1 \otimes 1 ) \otimes I_n + A \in \cl{C}_n^{\tilde{\tau}}$. By definition, there exist $E_{\varepsilon} \in \cl{F}(\cl{S})$ and $F_{\varepsilon} \in \cl{F}(\cl{T})$ for which $ A_{\varepsilon} \in M_n ( E_{\varepsilon} \otimes_{\tau} F_{\varepsilon} )^+$. Let $E = \cap_{\varepsilon > 0} E_{\varepsilon} \in \cl{F}(\cl{S})$ and $F = \cap_{\varepsilon > 0} F_{\varepsilon} \in \cl{F}(\cl{T})$, then by functorial property of $\tau$, for each $\varepsilon > 0$, $M_n( E \otimes_{\tau} F )^+ \subsetneq M_n( E_{\varepsilon} \otimes_{\tau} F_{\varepsilon} )^+$. Finally, since $E \otimes_{\tau} F$ defines an operator system, as $\varepsilon \to 0$ we see that $A_{\varepsilon} \to A \in M_n( E \otimes_{\tau} F )^+$. Consequently, $1 \otimes 1$ is an Archimedean matrix order unit; and $(\cl{S} \otimes  \cl{T}, \cl{C}_n^{\tilde{\tau}}, 1 \otimes 1 ) $ is an operator system. 	

It remains to show the (T2) and (T3) properties. Given $P = (p_{ij}) \in M_n(\cl{S})^+$ and $Q = (q_{st}) \in M_m(\cl{T})^+$, by choosing $E$ and $F$ to be the spans of $p_{ij}$'s and $q_{st}$'s, we have $P \otimes Q \in M_{nm}( E \otimes_{\tau} F )^+$. This shows that $ M_n(\cl{S})^+ \otimes M_m(\cl{T})^+ \subset \cl{C}_{nm}^{\tilde{\tau}}$. 	 
For (T3), we show further that it is functorial. Suppose $\phi \colon \cl{S}_1 \to \cl{S}_2$ and $\psi \colon \cl{T}_1 \to \cl{T}_2$ are completely positive maps. If $A \in \cl{C}_k^{\tilde{\tau}}$, then there are $E_1 \in \cl{F}(\cl{S})$ and $F_1 \in \cl{F}(\cl{T})$ such that $A \in M_k ( E_1 \otimes_{\tau} F_1 )^+$. Let $E_2$ and $F_2$ denote the ranges of $\phi$ and $\psi$, respectively. By functorial property of $\tau$, the map $\phi \otimes \psi|_{E_1 \otimes_{\tau} F_1} \colon E_1 \otimes_{\tau} F_1 \to E_2 \otimes_{\tau} F_2$ is completely positive. In particular, $( \phi \otimes \psi )^{(k)} (A) \in M_k( E_2 \otimes_{\tau} F_2 )^+$. Therefore, $\phi \otimes \psi$ is completely positive and $\tilde{\tau}$ is functorial. 
\end{proof}

\begin{prop} 
$\tilde{\tau}$ preserves injectivity, symmetry, and projectivity. 
\end{prop}

\begin{proof}
Let $\tau$ be injective, $\cl{S}_1 \subset \cl{S}$ and $\cl{T}_1 \subset \cl{T}$ be operator subsystems, and $A \in M_n( \cl{S} \otimes_{\tilde{\tau}} \cl{T} )^+ \cap M_n ( \cl{S}_1 \otimes \cl{T}_1 )$. By definition, $A \in M_n( E \otimes_{\tau} F)^+$ for some finite dimensional operator subsystems $E \subset \cl{S}$ and $F \subset \cl{T}$. Hence  $E \cap \cl{S}_1$ and $F \cap \cl{T}_1$ are finite dimensional operator subsystems of $\cl{S}_1$ and $\cl{T}_1$ respectively. By injectivity of $\tau$, 
	\begin{equation*}
		A \in M_n( E \otimes_{\tau} F)^+ \cap M_n( \cl{S}_1 \otimes \cl{T}_1 ) = M_n( ( E \cap \cl{S}_1 ) \otimes_{\tau} (F \cap \cl{T}_1) )^+.
	\end{equation*}
This shows that $A \in M_n( \cl{S}_1 \otimes_{\tilde{\tau}} \cl{T}_1 )^+$, and $\cl{S}_1 \otimes_{\tilde{\tau}} \cl{T}_1$ is complete order included in $\cl{S} \otimes_{\tilde{\tau}} \cl{T}$, proving $\tilde{\tau}$ is injective. 

Let $\tau$ be symmetric, and $\phi \colon \cl{S} \otimes_{\tilde{\tau}} \cl{T} \to \cl{T} \otimes_{\tilde{\tau}} \cl{S}$ be the map $x \otimes y$ to $y \otimes x$. If $u \in M_n( \cl{S} \otimes_{\tilde{\tau}} \cl{T} )^+$, then $u \in M_n( E \otimes_{\tau} F )^+$, for some finite dimensional $E$ and $F$; so $\phi^{(n)}(u) \in M_n( F \otimes_{\tau} E)^+ \subset M_n( \cl{S} \otimes_{\tilde{\tau}} \cl{T} )^+$ and $\tilde{\tau}$ is symmetric. 

Suppose $\tau$ is projective, and $q \colon \cl{S} \to \cl{V}$ and $\rho \colon \cl{T} \to \cl{W}$ are complete quotient maps. We claim that every $U \in M_n( \cl{V} \otimes_{\tilde{\tau}} \cl{W} )^+$ can lift to a positive $\tilde{U} \in M_n( \cl{S} \otimes_{ \tilde{\tau} } \cl{T} )$. Since $U \in M_n( X \otimes_{\tau} Y)^+$, for some $X \in \cl{F}(\cl{V})$ and $Y \in \cl{W}( \cl{T} )$, using projectivity of $\tau$, there is $\tilde{U} \in M_n (E \otimes_{\tau} F)^+$ for which $E \in \cl{F}(\cl{S})$, $F \in \cl{F}(\cl{T})$ and $q \otimes \rho (\tilde{U}) = U$. Therefore, $\tilde{\tau}$ is projective.
\end{proof}

\begin{remark}
We remark that $\tilde{\tau}$ indeed extends $\tau$. If $\cl{S}$ and $\cl{T}$ are finite dimensional, then $\cl{C}_n^{\tilde{\tau}} ( \cl{S}, \cl{T} ) = M_n( \cl{S} \otimes_{\tau} \cl{T} )^+$ by functorial property of $\tau$, thus $\cl{S} \otimes_{\tau} \cl{T} = \cl{S} \otimes_{\tilde{\tau}} \cl{T}$. 
\end{remark}

\begin{lemma}
Let $\tau$ be a symmetric tensor product structure. Then $\tau$ is left projective (resp. injective) if and only if it is right projective (resp. injective), if and only if it is projective (resp. injective). 
\end{lemma}

\begin{proof}
Let $q \colon \cl{S} \to \cl{R}$ be a complete quotient map. Then this commuting diagram
		\begin{equation*}
		\xymatrix{			
					\cl{S} \otimes_{\tau} \cl{T} 	\ar[r]	\ar[d]_{q \otimes id}	&			\cl{T} \otimes_{\tau} \cl{S}	\ar[d]^{id \otimes q} \\
					\cl{R} \otimes_{\tau} \cl{T}			&			\cl{T} \otimes_{\tau} \cl{R}	\ar[l]
		}
	\end{equation*}
asserts the equivalent condition. Similarly, if $\cl{R}$ is a operator subsystem of $\cl{S}$, then 
		\begin{equation*}
		\xymatrix{			
					\cl{S} \otimes_{\tau} \cl{T} 	\ar[r]	\ar[d]_{\iota \otimes id}	&			\cl{T} \otimes_{\tau} \cl{S}	\ar[d]^{id \otimes \iota} \\
					\cl{R} \otimes_{\tau} \cl{T}			&			\cl{T} \otimes_{\tau} \cl{R}	\ar[l]
		}
	\end{equation*}
shows that $\tau$ is left injective if and only if it is right injective. 
\end{proof}

Henceforth, given $\tau$ on finite dimensional operator systems, there is no ambiguity to say $\tau$ defines a tensor product structure on arbitrary operator systems. By this natural extension,  we see that $\gamma_{soh} $ defines a symmetric tensor product structure on operator systems. The cone $M_n( \cl{S} \otimes_{\gamma_{soh}} \cl{T} )^+$ is precisely the set of $u \in \cl{S} \otimes M_n(\cl{T})$ so that $\hat{u} \colon E^d \to M_n(F)$ factors through $SOH$ approximately, for some $E \in \cl{F}(\cl{S})$ and $F \in \cl{F}(\cl{T})$. 

Some questions about $\gamma_{soh}$ remain. We do not know if it is
injective or projective. By the lemma above, it suffices to check
these properties on one side. We do not yet know if $\gamma_{soh}$ is
distinct from the commuting tensor or any of the symmetric tensors that arise
from two-sided inclusions into the maximal tensor products of the
injective envelope or the C*-envelope. 

\section*{Acknowledgements}

We woud like to thank V.~P. Gupta, P.~Luthra and A.~S. Kavruk for their comments and observations.

\end{document}